\documentstyle[11pt]{article}
\begin{document}

\begin{center}
{\large \bf PERIODIC SOLUTIONS OF PERIODICALLY PERTURBED PLANAR
AUTONOMOUS SYSTEMS: A TOPOLOGICAL APPROACH $^\dagger$}
\end{center}
\begin{center}
Mikhail Kamenskii$^1$, Oleg Makarenkov$^1$ and Paolo Nistri$^2$
\end{center}
\begin{center}
$^1$Department of Mathematics,\\
 Voronezh State University,
Voronezh, Russia.\\
E-mail: Mikhail\verb+@+kam.vsu.ru,$\;\;$ E-mail: omakarenkov\verb+@+kma.vsu.ru\\
$^2$ Dipartimento di Ingegneria dell'Informazione,\\
Universit\`a di Siena, 53100 Siena, Italy.\\
E-mail: pnistri\verb+@+dii.unisi.it
\end{center}

\date{}
\def\thefootnote{\fnsymbol{footnote}}
\footnotetext[2]{Supported by the national research project MIUR:
``Nonlinear and optimal control: geometrical and topological
methods'', by GNAMPA-CNR, by RFBR grants 02-01-00189, 05-01-00100,
by U.S.CRDF - RF Ministry of Education grant VZ-010 and by
President of Russian Federation Fellowship for Scientific Training
Abroad}
\def\thefootnote{\fnsymbol{footnote}}

\newtheorem{theorem}{Theorem}
\newtheorem{remark}{Remark}
\newtheorem{lemma}{Lemma}
\newtheorem{corollary}{Corollary}
\newtheorem{proposition}{Proposition}
\newtheorem{definition}{Definition}
\def\thedefinition{\arabic{definition}}
\def\thelemma{\arabic{lemma}}
\def\theremark{\arabic{remark}}
\def\theproposition{\arabic{proposition}}
\def\thecorollary{\arabic{corollary}}
\def\thetheorem{\arabic{theorem}}
\def\vs{\vskip12pt}
\def\qed{\hfill $\square$}
\def\squar{\vbox{\hrule\hbox{\vrule height 6pt \hskip
6pt\vrule}\hrule}}
\def\sqr{{\unskip\nobreak\hfil\penalty50\hskip2em\hbox{}\nobreak\hfil
{\squar} \parfillskip=0pt\finalhyphendemerits=0\par}}
\newenvironment{proof}{\noindent{\bf Proof}:\ \ }{\sqr\par\vs}
\def\inte{\int\limits}     \let\qq=\qquad       \let\q=\quad
\let\w=\widetilde          \let\wh=\widehat     \let\mx=\mbox
\let\ol=\overline          \let\D=\Delta        \let\d=\delta
\let\e=\epsilon            \let\g=\gamma        \def\mm{{-\!\circ}}
\let\G=\Gamma              \let\ba=\beta        \def\dist{{\fam0 dist\,}}
\let\a=\alpha              \let\th=\theta       \let\nn=\nonumber
\let\s=\sigma              \let\O=\Omega      \def\L{{\cal L}}
\def\N{\mbox{\bf N}}       \def\bp{\mbox{\bf P}} \let\emp=\emptyset
\def\Z{\mbox{\bf Z}}       \def\ind#1{\mathop{#1}\limits}
\def\C{\mbox{\bf C}}       \let\l=\lambda       \let\sm=\setminus
\let\o=\omega              \def\meas{{\fam0 meas\,}}
\def\nor#1{{\left\|\,#1\,\right\|}}             \let\bcup=\bigcup
\def\R{\mbox{\bf R}}       \let\vf=\varphi    \def\sca#1#2{\langle #1,#2\rangle}
\def\Re{{\rm Re\,}}         \def\Im{{\rm Im\,}}        \def\DD{{\cal D}}
\let\ds=\displaystyle      \let\lra=\longrightarrow
\def\dfrac#1#2{\ds{#1\over #2}} \let\y=\eta
\let\sb=\subset            \def\J{{\cal J}}
\def\qed{\hfill $\squar$}
\def\squar{\vbox{\hrule\hbox{\vrule height 6pt \hskip
6pt\vrule}\hrule}}
\def\toto{{\begin{array}{c}\rightarrow\\*[-2ex]
             \rightarrow \end{array}}}
\catcode`\@=11
\def\section{\@startsection {section}{1}{\z@}{-3.5ex plus-1ex minus
-.2ex}{2.3ex plus.2ex}{\bf}}

\noindent
\begin{center}
{\bf Abstract}
\end{center}
{\scriptsize Aim of this paper is to investigate the existence of
periodic solutions of a nonlinear planar autonomous system having
a limit cycle $x_0$ of least period $T_0>0$ when it is perturbed
by a small parameter, $T_1-$periodic, perturbation. In the case
when $T_0/T_1$ is a rational number $l/k$, with $l, k$ prime
numbers, we provide conditions to guarantee, for the parameter
perturbation $\varepsilon>0$ sufficiently small, the existence of
$klT_0-$ periodic solutions $x_\varepsilon$ of the perturbed
system which converge to the trajectory $\tilde x_0$ of the limit
cycle as $\varepsilon\to 0$. Moreover, we state conditions under
which $T=klT_0$ is the least period of the periodic solutions
$x_\varepsilon$. We also suggest a simple criterion which ensures
that these conditions are verified. Finally, in the case when
$T_0/T_1$ is an irrational number we show the nonexistence,
whenever $T>0$ and $\varepsilon>0$, of $T-$periodic solutions
$x_\varepsilon$ of the perturbed system converging to $\tilde
x_0$. The employed methods are based on the topological degree
theory.}



\vspace{4mm} \noindent
\section {Introduction}\label{in}

\vspace{2mm} \noindent The paper is devoted to the study of the
existence and of the behaviour as $\varepsilon \to 0$ of periodic
solutions of a perturbed autonomous differential system in ${\rm
R}^2$ of the form

\begin{equation}\label{2}
\dot x = \psi(x) + \varepsilon \phi(t,x),
\end{equation}
where $\psi:{\rm R}^2\to {\rm R}^2$ is a twice continuously
differentiable function, $\phi:{\rm R}\times{\rm R}^2\to {\rm
R}^2$ is continuous and  $T_1$-periodic with respect to the first
variable and $\varepsilon>0$ is a small parameter.

\noindent To be specific, we assume that at $\varepsilon=0$ the
autonomous system
\begin{equation}\label{1}
  \dot x = \psi(x),
\end{equation}
has a limit cycle $x_0=x_0(t), \ t\ge0,$ of least period  $T_0$
satisfying the following condition

\vspace{3mm} \noindent

(A$_0$)- the linear system

\begin{equation}\label{ls}\dot y=\psi'(x_0(t))y
\end{equation}

does not have $2T_0$-periodic solutions linearly independent with
$\dot x_0(t).$

\vspace{4mm} \noindent Furthermore, we assume

(A$_1$)- there exist $l,k\in{\rm N}$ such that
$\dfrac{T_0}{T_1}=\dfrac{l}{k},$ with $l$ and $k$ prime numbers.

\vspace{4mm} \noindent This paper addresses the following
problems:

\vspace{2mm} \noindent 1. {\it To provide conditions on the
function $\phi$ which guarantee the existence of $\varepsilon_*>0$
such that system (\ref{2}) has a $T$-periodic solution
$x_\varepsilon,$ with $T=T(T_0,T_1)$ and
$\varepsilon\in(0,\varepsilon_*)$ satisfying the property
$$
x_\varepsilon(t)\to \tilde x_0 \ as \ \varepsilon\to 0,\quad
\mbox{whenever} \quad t\in[0,T],
$$
where $\tilde x_0= \{x\in {\rm R}^2: x=x_0(t), \, t\in[0,T_0]\}$
is the trajectory of the limit cycle of (\ref{1})}.

\vspace{2mm} \noindent 2. {\it To find an explicit estimation of
$\varepsilon_*>0.$}

\vspace{2mm} \noindent 3. {\it To investigate the existence of
periodic solutions for system (\ref{2}) in the case when condition
(A$_1$) is not verified, i.e. when $\dfrac{T_0}{T_1}$ is an
irrational number.}

\vspace{4mm} \noindent Several contributions have been made toward
solving the problem of the existence and the nature of periodic
solutions of a periodically perturbed autonomous system which are
close to the periodic solution of the autonomous system. We refer
to the papers \cite{L1}, \cite{L2}, \cite{Levinson} and to the
memoir \cite{L3} for first and second order perturbed autonomous
systems respectively. The employed methods are based on
perturbation and implicit function techniques which require more
regularity on the function $\phi$ of that required in this paper.
In \cite{L1}, under assumptions similar to ($A_0$) and ($A_1$) the
author provides sufficient conditions for the existence and
stability of periodic solutions of a perturbed autonomous system
in ${\rm R}^n$.  We would like to point out that in this paper we
are interested in periodic solutions of general form, that is
without any condition on their rate of convergence to $\tilde x_0$
with respect to $\varepsilon.$ In such a generality the uniqueness
of the periodic solutions is not guaranteed as in \cite{L1}, where
a special class of periodic solutions, converging to $\tilde x_0$
at the same rate of $\varepsilon,$ is considered. Similar results
for the existence of $D-$periodic functions, namely for functions
whose derivative is periodic, can be found in \cite{F1} and
\cite{F2}. Furthermore, in \cite{B1}, \cite{B2} and \cite{B3}
analogous existence and stability results are established under
different assumptions. Specifically, in \cite{B1}, under the
assumption that $1$ is the only characteristic multiplier of the
linearized system, which is a $q$-$th$ root of unity, for $q\in
{\bf N}$, it is proved that for any $p\in {\bf N}$ there exists a
parametrized family of periods and corresponding periodic
solutions of the perturbed autonomous system such that at
$\varepsilon=0$ the period is given by $q\,T_0/p$. In \cite{B2}
and \cite{B3} the behaviour, with respect to the period of the
disturbance, of the periodic surface described by a functional
equation and associated to the perturbed autonomous system is
studied. Finally, in \cite{Furi} and \cite{Furi1} existence and
multiplicity results of periodic solutions of a periodically
perturbed autonomous second order equation, defined on a manifold,
were proved by means of topological methods. An example of
periodically perturbed autonomous systems arising in Mathematical
Biology is given in \cite{G1}.

\noindent An other approach to the study of the existence of
periodic solutions of perturbed autonomous systems based on
topological degree and index theory is presented in \cite{CMZ} and
\cite{AGL} respectively. In these papers the perturbation is not
necessarily small, in the sense that it does not contain small
perturbation parameters, the period $T$ is fixed and the
autonomous system may have periodic orbits of period less than
$T.$

\vspace{3mm}\noindent In this paper to solve the proposed problems
we use a different approach which was previously introduced and
successfully employed by the authors in \cite{kmn1}, \cite{kmn2}
and \cite{kmn3} to treat the existence problem of periodic
solutions of perturbed nonautonomous systems of prescribed period.
Specifically, in \cite{kmn1} we  consider the nonautonomous system
of differential equations described by
$$
\dot x=\psi(t,x)+\varepsilon\phi(t,x), \eqno{(NA)}
$$
where $\phi,\ \psi:{\rm R}\times{\rm R}^n\to {\rm R}^n$ are
continuously differentiable, $T$-periodic with respect to time
$t$, functions and $\varepsilon$ is a small positive parameter.

\noindent The proposed approach is based on the linearized system
associated to (NA)
$$
\dot y=\frac{\partial\psi}{\partial x}(t,\Omega(t,0,\xi))y +
\phi(t,\Omega(t,0,\xi)), \eqno{(LNA)}
$$
where $\xi\in {\rm R}^n$ and $\Omega(\cdot,t_0,\xi)$ denotes the
solution of (NA) at $\varepsilon=0$ satisfying $x(t_0)=\xi$.
Specifically, consider the change of variable
$$
z(t)= \Omega(0,t,x(t)), \eqno{(CV)}
$$
and the solution $\eta(\cdot,s,\xi)$ of (LNA) such that $y(s)=0$.
If there exists a bounded open set $U\subset {\rm R}^n$ such that
$\Omega(T,0,\xi)=\xi$ for any $\xi\in\partial U$, and
$\eta(T,s,\xi)-\eta(0,s,\xi)\not=0,$ for any $s\in[0,T],$  and any
$\xi\in\partial U.$ Then (NA) has a $T$-periodic solution for
$\varepsilon>0$ sufficiently small provided that
$\Gamma(\eta(T,0,\cdot), U)\not=0$. Here $\Gamma(F,U)$ denotes the
rotation number of a continuous map $F:\overline U\to {\rm R}^2$.
Observe that $\Gamma(F,U)$ coincides with ${\rm deg}(F,U,0)$, the
topological degree of $F$ at $0$ relative to $U$. In what follows,
we will omit the point $0$ in the notation for the topological
degree.

\noindent The advantage of the proposed approach as compared with
the classical averaging method, which is one of the most useful
tools for treating the existence problem of periodic solutions of
nonautonomous periodic systems, mainly consists in the fact that
in order to use this second method for establishing the existence
of periodic solutions in perturbed systems of the form (NA) one
must assume that the change of variable (CV) is $T$-periodic with
respect to $t$ for every $T$-periodic function $x$ such that
$\Omega(0,t, x(t))\in U,$ for any $t\in [0,T]$, instead that only
on the boundary of the bounded open set $U$.

\noindent The same assumption is necessary in vibrational control
problems, \cite{b} and \cite{sch}, to reduce the considered system
to the standard form for applying the averaging method. For an
extensive list of references on this topic see \cite{bullo}.

\noindent Our approach has been also employed in \cite{kmn2} to
prove the existence of periodic solutions for a class of first
order singularly perturbed differential systems.

\noindent Furthermore, in \cite{kmn3} we have considered two
nonlinear small periodic perturbations of $\psi(t,x)$ with
multiplicative different powers of $\varepsilon>0$ and we have
proved the existence of $T-$periodic solutions of the resulting
system. We have also showed the presence of the so-called
frequency pulling phenomenon in the case of a special class of
planar autonomous systems when they are perturbed by periodic
terms with period close to that of the periodic solution of the
autonomous system, whose existence is assumed.

\noindent The aim of this paper is at solving problems 1-3 by
using the approach outlined above, extending in this way its
application to the study of the existence of periodic solutions
and of their behaviour for general, periodically perturbed, planar
autonomous systems having a limit cycle. In particular, we are
interested in investigating the relationship between the period of
the limit cycle and that of the nonautonomous perturbations in
order to have periodic solutions of (\ref{1}) with the period
expressed in terms of the previous ones.

\noindent To the best knowledge of the authors the present paper
treats for the first time the problem of evaluating
$\varepsilon_*>0$, i.e. problem 2. stated before, for the
application of the averaging principle via topological methods.

 \noindent
The paper is organized as follows. In section 2 we show our main
result: theorem 1 which provides sufficient conditions to ensure,
for $\varepsilon>0$ sufficiently small, the existence of
$klT_0-$periodic solutions $x_\varepsilon$ to (\ref{1}) converging
to $\tilde x_0$ as $\varepsilon\to 0$. In theorem 2 we show, under
a mild extra assumption, that $klT_0$ is the least period.

\noindent In section 3 we present a simple method to ensure that
the conditions of theorem 1 are satisfied.

\noindent Finally, in Section 4 we establish a nonexistence
result: theorem 4, in the case when $T_0/T_1$ is irrational.
Precisely, we show that, for any $T>0$ and $\varepsilon>0$, there
is not a $T-$periodic solution $x_\varepsilon$ of (\ref{1}) such
that $x_\varepsilon \to \tilde x_0$ as $\varepsilon\to 0$. An
example illustrating the existence result is also provided.

\section{The main result}\label{os}

Throught this section we assume that $T=kl T_0$ and we will denote
by $F^\prime_{(i)}$ the derivative of the function $F$ with
respect to the $i$-$th$ argument. Let $x(t)=\Omega(t,0,\xi)$ be
the solution of system (\ref{1}) satisfying $x(0)=\xi.$ Consider
the following auxiliary system of linear ordinary differential
equations
\begin{equation}\label{vsp}
  \dot y = \psi'(\Omega(t,0,\xi))y+\phi(t,\Omega(t,0,\xi)).
\end{equation}
and let $y(t)=\eta(t,s,\xi)$ the solution of (\ref{vsp})
satisfying $y(s)=0.$ Observe that if $\psi=0$ then
$$
\eta(T,s,\xi)-\eta(0,s,\xi)=\int\limits_0^{T} \phi(\tau,\xi)d\tau.
$$
Therefore, in this case, the function $\dfrac{1}{T}\left(
\eta\left(T,s,\xi\right)-\eta(0,s,\xi)\right)$ is the average on
the interval $[0,T]$ of the function $\phi$ with respect to the
first variable. This function is the basis of the classical
averaging method, one of the most relevant tools to investigate
the existence of periodic solution of (NA) when $\psi=0$ (see
\cite{bog}).

\vspace{3mm}\noindent We can prove the following preliminary
result.

\begin{lemma}\label{le1}
$$
\eta(t,s,\xi)  =  \Omega'_{(3)}(t,0,\xi)\int\limits_s^t
\Phi(\tau,\xi) d\tau,
$$
where
\begin{equation}\label{dp}
 \Phi(t,\xi)=\Omega'_{(3)}(0,t,\Omega(t,0,\xi))\phi(t,\Omega(t,0,\xi)).
\end{equation}
\end{lemma}
\begin{proof}
 Observe that the matrix $\Omega'_{(3)}(t,0,\xi)$ is the
fundamental matrix, satisfying $\Omega'_{(3)}(0,0,\xi)=I$, for the
linear system
$$
  \dot y =\psi'(\Omega(t,0,\xi))y.
$$
Furthermore,
$\left(\Omega'_{(3)}(t,0,\xi)\right)^{-1}=\Omega'_{(3)}(0,t,\Omega(t,0,\xi)),$
in fact, by deriving with respect to $\xi$ the identity
$$
\Omega(0,t,\Omega(t,0,\xi))=\xi,\; \mbox{whenever} \; \xi\in{\rm
R}^2,
$$
we obtain
\begin{equation}\label{f12}
\Omega'_{(3)}(0,t,\Omega(t,0,\xi))\Omega'_{(3)}(t,0,\xi)=I,\;
\mbox{whenever} \; \xi\in{\rm R}^2.
\end{equation}
Therefore, by the variation of constants formula for the
nonhomogeneous system (\ref{vsp}) we obtain
\begin{eqnarray}
  \eta(t,s,\xi) & = & \int\limits_s^t
\Omega'_{(3)}(t,0,\xi)\left(\Omega'_{(3)}(\tau,0,\xi)\right)^{-1}\phi(\tau,
\Omega(\tau, 0,\xi))d\tau = \nonumber\\
      & = & \Omega'_{(3)}(t,0,\xi)\int\limits_s^t \Phi(\tau,\xi) d\tau.\nonumber
\end{eqnarray}
\end{proof}

\noindent Since the trajectory $\tilde x_0$ is a Jordan curve, its
interior $U$ is a bounded, simply connected, open set of ${\rm
R}^2$. In order to prove, by the proposed approach, the existence
of periodic solutions to (\ref{2}) converging to $\tilde x_0$ as
$\varepsilon\to 0,$ we introduce a family of open sets
$W_\gamma(U)$ as follows
$$
  W_\gamma(U)=\left\{\begin{array}{ll} U\backslash \left(\partial
  U+|\gamma|B\right) & {\rm if\ }\gamma<0, \\
  U\cup \left(\partial
  U+|\gamma|B\right) & {\rm if\ }\gamma>0,
  \end{array}\right.
$$
where $B\subset{\bf R}^2$ is the open unit ball, thus we have
\begin{equation}\label{conv11}
B_\gamma(\partial U):= ((W_\gamma(U)\cup U)\backslash
W_\gamma(U))\bigcup \,((W_\gamma(U)\cup U)\backslash U)\to \tilde
x_0\quad {\rm as\ }\gamma \to 0,
\end{equation}
where $\partial U=\tilde x_0$.

\vspace{3mm} \noindent
\begin{definition}\label{def1}
Define the following positive constants
\begin{eqnarray}
  M_\gamma&=&\max_{t\in[0,T],\ \xi\in \overline{B_\gamma(\partial U)}}
  \|\Phi(t,\xi)\|,\nonumber\\
  M_\gamma'&=&\max_{t\in[0,T],\ \xi\in \overline{B_\gamma(\partial U)}}
  \|\Phi'_{(2)}(t,\xi)\|,\nonumber\\
  L_\gamma'&=&\max_{\xi\in \overline{B_\gamma(\partial U)}}\|
  \Omega'_{(3)}(T,0,\xi)\|,\nonumber\\
  L_\gamma''&=&\max_{\xi\in \overline{B_\gamma(\partial U)}}\|
  \Omega''_{(3)(3)}(T,0,\xi)\|,\nonumber\\
  K_0&=&\min_{s\in[0,T],\ \xi\in \partial U}\|\eta(T,s,\xi)-
  \eta(0,s,\xi)\|,\nonumber\\
  K_\gamma&=&\min_{\xi\in\partial  W_\gamma(U)}\|\xi-\Omega(T,0,\xi)\|.
  \nonumber
\end{eqnarray}

\vspace{3mm} \noindent Moreover, let $\gamma_0>0$ such that system
(\ref{1}) does not have $T_0$-periodic solutions with initial
condition belonging to the open set $B_{-\gamma_0}(\partial
U)\bigcup B_{\gamma_0}(\partial U).$
\end{definition}

\noindent Observe that condition (A$_0$) guarantees the existence
of the constant $\gamma_0$ in the previous definition.

\vspace{4mm} \noindent We can now formulate the main result of the
paper.

\begin{theorem}\label{th1}
Assume ($A_0$), ($A_1$) and

{\rm (A$_2$)-} $\quad \eta(T,s,\xi)-\eta(0,s,\xi)\not=0,\
\mbox{for any} \ s\in[0,T],\ \mbox{and any}\ \xi\in\partial U,$

{\rm (A$_3$)-} $\quad {\rm deg}(\eta(0,T,\cdot),U)\not=1.$

\vspace{2mm} \noindent Then for every $-\gamma_0<\gamma<\gamma_0,\
\gamma\not=0$ and for
\begin{equation}\label{est}
0<\varepsilon<\min\left\{\frac{K_0}{T^2 M_\gamma
\left(M'_\gamma+\sqrt{2}M_\gamma L''_\gamma +M'_\gamma
L_\gamma'\right)},\; \frac{K_\gamma}{T M_\gamma
(1+L_\gamma'}\right\}:=\varepsilon_\gamma
\end{equation}
system (\ref{2}) has a $T$-periodic solution $x_\varepsilon$
belonging to the set
$$
\left\{x\in {\rm R}^2:\, x=\Omega(t,0,\xi),\ t\in [0,T], \ \xi\in
B_\gamma(\partial U)\right\}.
$$
\end{theorem}

\noindent To prove this theorem we need the following lemma.

\begin{lemma}\label{le2} Assume (A$_0$). Then there exists
$\gamma_1>0$ such that for $-\gamma_1<\gamma<\gamma_1$ and $\
\gamma\not=0$ we have
$$
{\rm deg}(I-\Omega(T,0,\cdot),W_\gamma(U))=1.
$$
\end{lemma}
\begin{proof}
Condition (A$_0$) ensures that the characteristic multiplier of
 $\Omega'_{(3)}(T,0,x(0))$ different from $+1$ is not equal to
 $-1.$ Therefore the limit cycle $x_0$ is either asymptotically
 stable or unstable, thus in a sufficiently small neighborhood of
 $\tilde x_0$ the map $\xi \to \Omega(T,0,\xi)$ does not have fixed
 points different from those belonging to $\tilde x_0$. In particular,
\begin{equation}\label{ooo}
\xi \neq \Omega(T,0,\xi)\quad{\rm for\ any\ } \xi \in
\partial W_{\gamma}(U)
\end{equation}
and $|\gamma|\neq 0$ sufficiently small. Hence by Corollary~2 of
\cite{CMZ} we have that
$$
{\rm deg}(I-\Omega(T,0,\cdot), W_\gamma(U))={\rm deg}(\psi,
W_\gamma(U)),
$$
On the other hand from (\ref{ooo}) we have that
$$
{\rm deg}(\psi, W_\gamma(U))={\rm deg}(\psi, U)
$$
and since the vector field $\psi$ is tangent to $\partial U$ at
any point, we get (see, for example, Theorem 2.3 of \cite{KPPZ})
$$
{\rm deg}(\psi, U)=1.
$$
 \end{proof}
The following result is a direct consequence of Lemma 2.
\begin{corollary}\label{cor1}
Under condition (A$_0$) we have that
\begin{equation}\label{cor2}
  {\rm deg}(I-\Omega(T,0,\cdot),W_\gamma(U))=1, \; \mbox{for every}\;
  -\gamma_0<\gamma<\gamma_0,\ \gamma\not=0.
\end{equation}
\end{corollary}

\begin{proof}
First we prove (\ref{cor2}) for $0<\gamma<\gamma_0$. By lemma
\ref{le2} there exists $0<\gamma_*<\gamma_0$  such that ${\rm
deg}(I-\Omega(T,0,\cdot),W_{\gamma_*}(U))=1.$ From definition~1 it
follows that the constant $\gamma_0$ has the property that the
vector field $I-\Omega(T,0,\cdot)$ is not degenerated on the
boundary of the set $W_\gamma(U)$ for any $0<\gamma<\gamma_0.$
Therefore, for any $0<\gamma<\gamma_0,$ we have
$$
 {\rm deg}(I-\Omega(T,0,\cdot),W_\gamma(U))=
 {\rm deg}(I-\Omega(T,0,\cdot),W_{\gamma_*}(U))=1.
$$
The same arguments apply for $-\gamma_0<\gamma<0.$

\end{proof}

\noindent We are now in the position to prove theorem 1.

\vspace{2mm} \noindent
\begin{proof}
Denote by $C([0,T],{\rm R}^2)$ the Banach space of all the
continuous functions defined on the interval $[0,T]$ with values
in ${\rm R}^2,$ equipped with the sup-norm. Consider in (\ref{2})
the change of variable
\begin{equation}\label{zpn}
  x(t)=\Omega(t,0,z(t)).
\end{equation}
For every $z \in C([0,T],{\rm R}^2),$ (\ref{zpn}) defines uniquely
$x \in C([0,T],{\rm R}^2)$ with inverse given by
\begin{equation}\label{zp_}
  z(t)=\Omega(0,t,x(t)),\ t \in [0,T].
\end{equation}
Therefore, the function $x$ is the solution of the system
(\ref{2}) if and only if the function $z$ defined by (\ref{zp_})
satisfies the differential equation
\begin{eqnarray}\label{f9}
  \Omega'_{(1)}(t,0,z(t))+\Omega'_{(3)}(t,0,z(t))\, \dot z(t)=
  \varepsilon \phi(t,\Omega(t,0,z(t)))+ \psi(\Omega(t,0,z(t))).
\end{eqnarray}
By the definition of $\Omega(t,0,z(t))$ we have
\begin{equation}\label{f11}
  \Omega'_{(1)}(t,0,z(t))=\psi(\Omega(t,0,z(t))).
\end{equation}
Moreover, by using (\ref{f12}) and (\ref{f11}) we can rewrite
system (\ref{f9}) in the following form
\begin{equation}\label{1_}
  \dot z(t)=\varepsilon\, \Phi(t,z(t)),
\end{equation}
where $\Phi$ is defined by (\ref{dp}). Observe that (\ref{zpn})
and  (\ref{zp_}) define a homeomorphism between the solutions of
systems (\ref{2}) and  (\ref{1_}).

\noindent Consider an arbitrary $T$-periodic solution $x$ of
system (\ref{2}), we have
$$
  z(0)=\Omega(0,0,x(0))=x(0)=x(T)=\Omega(T,0,z(T)).
$$
Therefore the problem of the existence of $T$-periodic solutions
to system (\ref{2}) is equivalent to the problem of the existence
of zeros of the compact vector field $G_\varepsilon : C([0,T],{\rm
R}^2)\to C([0,T],{\rm R}^2)$ defined by
$$
G_\varepsilon (z)(t)=z(t)-\Omega(T,0,z(T))-\varepsilon\int
\limits_0^t\Phi(\tau,z(\tau))d\tau, \quad t\in[0,T].
$$
Define the set
\begin{equation}\label{defZ}
  Z=\left\{z\in {\rm C}([0,T],{\rm R}^2):\ z(t)\in B_\gamma(\partial U)\;
  \mbox{for any}\; t\in[0,T]\right\}.
\end{equation}
Consider the auxiliary compact vector field
$$
G_{1,\varepsilon}=I-A_{\varepsilon}: C([0,T],{\rm R}^2)\to
C([0,T],{\rm R}^2),
$$
where
$$
A_{\varepsilon}(z)(t)=\Omega(T,0,z(T))-
    \varepsilon\int \limits_0^{T}\Phi(\tau,z(\tau))d\tau, \quad
    t\in[0,T],
$$
Let us show that for any $\varepsilon>0$ satisfying (\ref{est})
the compact vector fields $G_\varepsilon$ and $G_{1,\varepsilon}$
are homotopic on the boundary of the set $Z.$ For this, define the
following homotopy $F_\varepsilon : [0,1]\times C([0,T],{\rm
R}^2)\to C([0,T],{\rm R}^2)$ joining the vector fields
$G_\varepsilon$ and $G_{1,\varepsilon}$:
$$
F_{\varepsilon}(\lambda,z)(t)=z(t)-\Omega(T,0,z(T))-\varepsilon\int
\limits_0^{\alpha(\lambda, t)}\Phi(\tau,z(\tau))d\tau, \quad
    t\in[0,T],
$$
where $\alpha(\lambda, t)=\lambda t+(1-\lambda)T$. Let us show
that for any $\varepsilon>0$ satisfying (\ref{est}) the homotopy
$F_\varepsilon$ does not vanish on the boundary of the set $Z.$
Assume the contrary, thus for some $\varepsilon>0$ satisfying
(\ref{est}) there exist $z_\varepsilon\in\partial Z$ and
$\lambda_\varepsilon\in[0,1]$ such that
\begin{equation}\label{f15}
  z_\varepsilon(t)=\Omega(T,0,z_\varepsilon(T))+\varepsilon\int
\limits_0^{\alpha(\lambda_\varepsilon, t)}
\Phi(\tau,z_\varepsilon(\tau))d\tau,\quad t\in[0,T].
\end{equation}
From the fact that $z_\varepsilon\in\partial Z$ it follows the
existence of $t_\varepsilon\in[0,T]$ such that
$z_\varepsilon(t_\varepsilon)\in\partial \left( B_\gamma(\partial
U) \right).$
By the definition of the set $B_\gamma(\partial U)$ either
\begin{equation}
\label{case1} z_\varepsilon(t_\varepsilon)\in\partial W_\gamma(U),
\end{equation}
or
\begin{equation}\label{case2}
z_\varepsilon(t_\varepsilon)\in\partial U.
\end{equation}
By using (\ref{est}) we have the following estimate
\begin{eqnarray}
& &
\left\|z_\varepsilon(t_\varepsilon)-\Omega(T,0,z_\varepsilon(t_\varepsilon))
\right\|  = \nonumber\\
& &
\left\|z_\varepsilon(t_\varepsilon)-\Omega(T,0,z_\varepsilon(T))+
 \Omega(T,0,z_\varepsilon(T))-\Omega(T,0,z_\varepsilon(t_\varepsilon))\right\|=
 \nonumber\\
& & =
\left\|\varepsilon\int\limits_0^{\alpha(\lambda_\varepsilon,t_\varepsilon)}
\Phi(\tau,z_\varepsilon(\tau))d\tau
  + \Omega(T,0,z_\varepsilon(T))-\Omega(T,0,z_\varepsilon(t_\varepsilon))
  \right\|\le\nonumber\\
& & \le \varepsilon T M_\gamma + \varepsilon L_\gamma' T
M_\gamma<K_\gamma,
\end{eqnarray}
which contradicts the definition of the constant $K_\gamma$ in the
case when (\ref{case1}) holds true. Thus (\ref{case1}) cannot
occur. Assume now (\ref{case2}), by lemma 1 and the fact that, in
this case, $z_\varepsilon(t_\varepsilon)=
\Omega(T,0,z_\varepsilon(t_\varepsilon))$ we have
\begin{eqnarray}
  & & \hskip-0.7cm \eta(T,\alpha(\lambda_\varepsilon,t_\varepsilon),
  z_\varepsilon(t_\varepsilon))-
  \eta(0,\alpha(\lambda_\varepsilon,t_\varepsilon),z_\varepsilon(t_\varepsilon))=
  \nonumber\\
  & & \hskip-0.7cm
=\left(\Omega'_{(3)}(T,0,z_\varepsilon(t_\varepsilon))-I\right)\int\limits_
{\alpha(\lambda_\varepsilon,t_\varepsilon)}^{T}\Phi(\tau,z_\varepsilon(t_\varepsilon))
d\tau+\int\limits_0^{T
}\Phi(\tau,z_\varepsilon(t_\varepsilon))d\tau=\nonumber\\
  & & \hskip-0.7cm =
\frac{z_\varepsilon(T)-z_\varepsilon(t_\varepsilon)}{\varepsilon}-\frac{
\Omega(T,0,z_\varepsilon(T))-\Omega(T,0,z_\varepsilon(t_\varepsilon))}
{\varepsilon}-\int\limits_0^{T}\Phi(\tau,z_\varepsilon(\tau))d\tau-\nonumber\\
  & & \hskip-0.7cm -
\left(I-\Omega'_{(3)}(T,0,z_\varepsilon(t_\varepsilon))\right)\int\limits_
{\alpha(\lambda_\varepsilon,t_\varepsilon)}^{T}\Phi(\tau,z_\varepsilon(t_\varepsilon))
d\tau+
\int\limits_0^{T }\Phi(\tau,z_\varepsilon(t_\varepsilon))d\tau =\nonumber\\
  & & \hskip-0.7cm
=\int\limits_{\alpha(\lambda_\varepsilon,t_\varepsilon)}^{T}\Phi(\tau,z_\varepsilon(\tau))
d\tau-\frac{
1}{\varepsilon}\left(\begin{array}{c}[\Omega'_{(3)}(T,0,\xi_{1\varepsilon})
(z_\varepsilon(T)-z_\varepsilon(t_\varepsilon))]_1\\

[\Omega'_{(3)}(T,0,\xi_{2\varepsilon})(z_\varepsilon(T)-z_\varepsilon(t_\varepsilon))]_2
  \end{array}\right)-\nonumber\\
  & & \hskip-0.7cm -
\left(I-\Omega'_{(3)}(T,0,z_\varepsilon(t_\varepsilon))\right)
\int\limits_{\alpha(\lambda_\varepsilon,t_\varepsilon)}^{T}\Phi(\tau,z_\varepsilon(t_\varepsilon))
d\tau
+\int\limits_0^{T}\Phi(\tau,z_\varepsilon(t_\varepsilon))d\tau-\nonumber\\
  & & \hskip-0.7cm
- \int\limits_0^{T}\Phi(\tau,z_\varepsilon(\tau))d\tau =
\int\limits_0^{\alpha(\lambda_\varepsilon,t_\varepsilon)}\Phi(\tau,z_\varepsilon(t_\varepsilon))d\tau-
\int\limits_{0}^{\alpha(\lambda_\varepsilon,t_\varepsilon)}\Phi(\tau,z_\varepsilon(\tau))d\tau-\nonumber\\
    & & \hskip-0.7cm-
\left(\begin{array}{c}[\Omega'_{(3)}(T,0,\xi_{1\varepsilon})\int
\limits_{\alpha(\lambda_\varepsilon,t_\varepsilon)}^{T}\Phi(\tau,z_\varepsilon(\tau))d\tau]_1\\

[\Omega'_{(3)}(T,0,\xi_{2\varepsilon})\int\limits_{\alpha(\lambda_\varepsilon,t_\varepsilon)}^{T}
\Phi(\tau,z_\varepsilon(\tau))d\tau]_2
\end{array}\right)+\nonumber\\
  & & \hskip-0.7cm
+\left(\begin{array}{c}[\Omega'_{(3)}(T,0,z_\varepsilon(t
_\varepsilon))\int\limits_{\alpha(\lambda_\varepsilon,t_\varepsilon)}^{T}\Phi(\tau,z_\varepsilon(\tau))d\tau]_1\\

[\Omega'_{(3)}(T,0,z_\varepsilon(t_\varepsilon))\int\limits_{\alpha(\lambda_\varepsilon,t_\varepsilon)}
^{T}\Phi(\tau,z_\varepsilon(\tau))d\tau]_2
\end{array}\right)-\nonumber\\
    & & \hskip-0.7cm-
\Omega'_{(3)}(T,0,z_\varepsilon(t_\varepsilon))\int\limits_{\alpha(\lambda_\varepsilon,t_\varepsilon)}^{T}
\Phi(\tau,z_\varepsilon(\tau))d\tau
+\,\Omega'_{(3)}(T,0,z_\varepsilon(t_\varepsilon))\int\limits_{\alpha(\lambda_\varepsilon,t_\varepsilon)}^{T}
\Phi(\tau,z_\varepsilon(t_\varepsilon))d\tau  ,\label{newr1}
\end{eqnarray}
where
$\xi_{i\varepsilon}\in[z_\varepsilon(t_\varepsilon),z_\varepsilon(T)]_i$
and $[v]_i, i=1,2,$ denotes the $i$-$th$ component of the vector
$v$. By using the following estimates
$$
\|z_\varepsilon(t_\varepsilon)-\xi_{i\varepsilon}\|\le
\|z_\varepsilon(t_\varepsilon)-z_\varepsilon(T)\|\le\left\|
\varepsilon\int\limits_{\alpha(\lambda_\varepsilon,t_\varepsilon)}^{T}\Phi(
\tau,z_\varepsilon(\tau))d\tau\right\|\le\varepsilon
  T M_\gamma \qquad \mbox{and}
$$
$$
\max_{t\in[0,T]}\left\|\Phi(t,z_\varepsilon(t_\varepsilon))-\Phi(t,z_\varepsilon
(t))\right\|\le M'_\gamma
\max_{t\in[0,T]}\|z_\varepsilon(t_\varepsilon)-z_\varepsilon(t)\|\le
\varepsilon T M_\gamma M'_\gamma $$ together with (\ref{est}) we
obtain from (\ref{newr1})
\begin{eqnarray}
& &
\left\|\eta(T,\alpha(\lambda_\varepsilon,t_\varepsilon),z_\varepsilon(t_\varepsilon))-
\eta(0,\alpha(\lambda_\varepsilon,t_\varepsilon),z_\varepsilon(t_\varepsilon))\right\|\le\nonumber\\
& & \le\varepsilon T^2 M_\gamma M'_\gamma
+\varepsilon\sqrt{2}\hskip0.1cm T^2(M_\gamma)^2
L_\gamma''+\varepsilon T^2 M_\gamma M_\gamma'
L_\gamma'=\nonumber\\
& & = \varepsilon T^2 M_\gamma\left(M'_\gamma+\sqrt{2} M_\gamma
L_\gamma''+M'_\gamma L_\gamma'\right)<K_0,
\end{eqnarray}
which contradicts the definition of the constant $K_0.$ Hence both
(\ref{case1}) and (\ref{case2}) cannot occur and so for
$\varepsilon\in(0,\varepsilon_\gamma)$ the homotopy
$F_\varepsilon$ does not vanish on the boundary of the set $Z$ and
so the compact vector fields $G_\varepsilon$ and
$G_{1,\varepsilon}$ are homotopic on the boundary of the set $Z.$
In particular,
\begin{equation}\label{inpart}
G_{1,\varepsilon}(z)\not= 0,\  \mbox{for any} \ z \in
\partial\left(Z\bigcap C_{const}([0,T],{\rm R}^2)\right)\subset\partial Z,
\ \mbox{and any} \ \varepsilon\in(0,\varepsilon_\gamma),
\end{equation}
where $C_{const}([0,T],{\rm R}^2)$ denotes the subspace of the
space $C([0,T],{\rm R}^2)$ consisting of all the constant
functions defined on the interval $[0,T]$ with values in ${\rm
R}^2$ and $\partial\left(Z\bigcap C_{const}([0,T],{\rm
R}^2)\right)$ is the relative boundary of the set $Z\bigcap
C_{const}([0,T],{\rm R}^2)$ with respect to the subspace
$C_{const}([0,T],{\rm R}^2).$
Therefore, by the reduction domain property of the topological
degree, taking into account that $A_{\varepsilon}(\partial
Z)\subset C_{const}([0,T],{\rm R}^2),$ we obtain
\begin{equation} \label{vr1}
  {\rm deg}_{C([0,T],{\rm R}^2)}(G_{1,\varepsilon},Z)=
    {\rm deg}_{C_{const}([0,T],{\rm R}^2)}\left(G_{1,\varepsilon},
       Z\bigcap C_{const}([0,T],{\rm R}^2)\right),
\end{equation}
whenever $\varepsilon\in(0,\varepsilon_\gamma).$ Observe, that
$z\in Z\bigcap C_{const}([0,T],{\rm R}^2)$ is a solution of the
equation $G_{1,\varepsilon}z=0$ if and only if $\xi=z$ is a
solution of the equation $Q_{\varepsilon}\xi=0,$ where
$Q_\varepsilon: {\rm R}^2\to {\rm R}^2$ is defined by
$$
Q_{\varepsilon}\xi=\xi-\Omega(T,0,\xi)-\varepsilon\int
\limits_0^{T}\Phi(\tau,\xi)d\tau.
$$
Therefore, we have
\begin{equation}\label{vr2}
  \begin{array}{c}
    {\rm deg}_{C_{const}([0,T],{\rm R}^2)}(G_{1,\varepsilon},Z\bigcap C_{const}([0,T],{\rm
R}^2))=
      {\rm deg}_{{\rm R}^2}(Q_{\varepsilon},B_\gamma(\partial U))=
      \\\\
     = {\rm sign}(\gamma)\left({\rm deg}_{{\rm R}^2}(Q_\varepsilon,W_\gamma(U))-{\rm deg}_{{\rm
     R}^2}(Q_\varepsilon,U)\right),\ \ \mbox{for any} \ \varepsilon\in(0,\varepsilon_\gamma).
  \end{array}
\end{equation}
From (\ref{est}), for any $\varepsilon\in(0,\varepsilon_\gamma),$
we have
$$
 \min_{\xi\in\partial  W_\gamma(U)}\|\xi-\Omega(T,0,\xi)\|=K_\gamma>
  \varepsilon T M_\gamma\ge \varepsilon
    \max_{\xi\in B_\gamma(\partial U)}
      \left\|\int\limits_0^T\Phi(\tau,\xi)d\tau\right\|
$$
and so
$$
  {\rm deg}_{{\rm R}^2}(Q_\varepsilon,W_\gamma(U))={\rm deg}_{{\rm R}^2}(Q_0,W_\gamma(U)),
   \quad \mbox{whenever} \quad \varepsilon\in(0,\varepsilon_\gamma).
$$
By corollary 1 and the previous equality we obtain
\begin{equation}\label{222}
  {\rm deg}_{{\rm R}^2}(Q_\varepsilon,W_\gamma(U))=1,
   \quad \mbox{whenever} \quad \varepsilon\in(0,\varepsilon_\gamma).
\end{equation}
Let us now calculate ${\rm deg}_{{\rm R}^2}(Q_\varepsilon,U).$ For
this, let $Q_{1,\varepsilon}: {\rm R}^2\to {\rm R}^2$ be defined
by
$$
Q_{1,\varepsilon}\xi=-\varepsilon\int
\limits_0^{T}\Phi(\tau,\xi)d\tau,\quad \mbox{for any}\quad \xi\in
Z\bigcap C_{const}([0,T],{\rm R}^2).
$$
From (A$_2$) we have that
$Q_{\varepsilon}\xi=Q_{1,\varepsilon}\xi,\ $ for any
$\xi\in\partial U,$ since $\xi=\Omega(T,0,\xi)$, and so for
$\varepsilon\in(0,\varepsilon_\gamma)$ we have
\begin{equation}\label{vr3}
{\rm deg}_{{\rm R}^2}(Q_{\varepsilon},U)= {\rm deg}_{{\rm
R}^2}(Q_{1,\varepsilon},U).
\end{equation}
Let us show that the compact vector fields $Q_{1,\varepsilon}$ and
$Q_{1,1}$ are homotopic on the boundary of the set $U$ for
$\varepsilon\in(0,\varepsilon_\gamma).$ For this, define the
linear homotopy $F_{1,\varepsilon}: [0,1]\times {\rm R}^2\to {\rm
R}^2$ as follows
$$
F_{1,\varepsilon}(\lambda,\xi)=-(\lambda\varepsilon+1-\lambda)\int
\limits_0^{T}\Phi(\tau,\xi)d\tau.
$$
Arguing by contradiction, assume that for some
$\hat\lambda\in[0,1],\ $ $\hat\xi\in\partial U$ and
$\hat\varepsilon\in(0,\varepsilon_\gamma)$ we have
$$
(\hat\lambda\hat\varepsilon+1-\hat\lambda)\int
\limits_0^{T}\Phi(\tau,\hat\xi)d\tau=0,
$$
and so
$$
\int \limits_0^{T}\Phi(\tau,\hat\xi)d\tau=0.
$$
By lemma 1 this is a contradiction with condition (A$_2$). Thus,
again by lemma 1
\begin{equation}\label{vr4}
{\rm deg}_{{\rm R}^2}(Q_{1,\varepsilon},U)= {\rm deg}_{{\rm
R}^2}(Q_{1,1},U)={\rm deg}_{{\rm R}^2}(\eta(0,T,\cdot),U).
\end{equation}
From (\ref{vr3}) and (\ref{vr4}) we obtain
\begin{equation}\label{res_}
  {\rm deg}_{{\rm R}^2}(Q_\varepsilon,U)=
    {\rm deg}_{{\rm R}^2}(\eta(0,T,\cdot),U),\ \mbox{for any} \ \varepsilon\in(0,\varepsilon_\gamma).
\end{equation}
Finally, taking into account (\ref{vr1})-(\ref{222}), (\ref{res_})
and condition (A$_3$) we get
\begin{eqnarray}
& & {\rm deg}_{C([0,T],{\rm R}^2)}(G_\varepsilon,Z)= {\rm
deg}_{C([0,T],{\rm R}^2)}(G_{1,\varepsilon},Z)= \nonumber\\
& & = {\rm sign}(\gamma)\left(1-{\rm deg}_{{\rm
R}^2}(\eta(0,T,\cdot),U)\right)\not=0,\nonumber
\end{eqnarray}
for any $\varepsilon\in(0,\varepsilon_\gamma)$. The solution
property of the topological degree ends the proof.
\end{proof}

\noindent {\bf Remark 1.}  Observe that ${\rm
deg}(\eta(0,T,\cdot),U)$ can be calculated by means of a well
known formula (see e. g. \cite{KZ}, p. 6). In this case condition
($A_3$) takes the form
\begin{eqnarray}
  {\rm deg}(\eta(0,T,\cdot),U) & = & \sum_{i=0}^{n-1}{\rm
  sign}([\eta(T,0,x_0(\cdot))]_1,\theta_{i+1},\theta_i)\cdot \nonumber \\
  & & \cdot ({\rm
  sign}[\eta(T,0,x_0(\theta_{i+1}))]_2-
  {\rm
  sign}[\eta(T,0,x_0(\theta_{i}))]_2)\not=1,\nonumber
\end{eqnarray}
where $[\cdot]_j,$ $j=1,2,$ denotes the $j-th$ component of the
vector field $\eta$ and $\theta_i,$ $i=0,1,...,n-1,$
$\theta_n=\theta_0,$ are the ordered roots of function
$\theta\mapsto[\eta(0,T,x_0(\theta))]_1.$

\noindent Conditions for the existence of solutions to (\ref{2})
in terms of the derivatives of the function
$[\eta(0,T,x_0(\cdot))]_1$ at the points $\theta_i$, up to a
suitable rotation, are given in \cite{L1}.

\

\noindent We now prove the following result.

\begin{theorem}\label{th2} Assume that there exists $\xi\in\tilde
x_0$ such that $T_1>0$ is the least period of the function $t \to
\phi(t,\xi).$ Assume condition ($A_1$), if
$\left\{x_\varepsilon\right\}_{\varepsilon\in(0,\varepsilon_*)}$
are $\hat T-$periodic solutions to (\ref{2}) converging to $\tilde
x_0$ as $\varepsilon\to 0$, then $klT_0 \le \hat T.$
\end{theorem}

\begin{proof}
Assume the contrary, thus there exists $\tilde T\in(0,klT_0)$ such
that for every $\varepsilon>0$ sufficiently small system (\ref{2})
has a $\tilde T$-periodic solution $x_\varepsilon$ satisfying
$$
\lim_{\varepsilon\to 0}\ x_\varepsilon(t)=x_0(t+w_0),
$$
for some $w_0\in[0,T_0]$, thus $\tilde T$ is a period of the
function $x_0(t).$ On the other hand, $T_0$ is the least period of
$x_0(t)$ and so we obtain $\tilde T=n_0 T_0$ for some $n_0\in{\rm
N}.$ Moreover, in (\ref{2}) we have that the function $t\to
\phi(t,x_\varepsilon(t))$ is $\tilde T$-periodic, since $\dot
x(t)$ is $\tilde T$-periodic, and so for any $t_0\in[0,T],$ we
have
$$
\phi(t_0,x_\varepsilon(t_0))=\phi(t_0+\tilde
T,x_\varepsilon(t_0+\tilde T))= \phi(t_0+\tilde
T,x_\varepsilon(t_0)).
$$
Due to our assumption and the arbitrarity of $t_0,$ by passing to
the limit as $\varepsilon\to 0$, we conclude that $\tilde T=n_1
T_1$ for some $n_1\in {\rm N}$. Therefore
$$
n_0 T_0 = n_1 T_1
$$
and so by ($A_1$) there exists $n_2\in{\rm N}$ such that $n_0= k l
n_2,$ contradicting the fact that $\tilde T<klT_0$.
\end{proof}

\section{A method to verify assumptions (A$_2$) and
(A$_3$) of theorem 1}\label{pr}

\noindent Although the topological degree (or rotation number) of
the vector field in (A$_3$) can be calculated by using the methods
from \cite{KZ}, as recalled in remark 1, we provide here an
alternative method to verify conditions (A$_2$) and (A$_3$) which
turns out to be useful and of simple application.

\noindent First, we introduce some notations. Recall that for the
vector $v\in{\rm R}^2$ we denote by $[v]_i$ its $i$-$th$
component, $i=1,2$, for $a\in{\rm R}^2$ we put
$$
a^{\bot}=\left(\begin{array}{c} -{[a]}_2 \\ \; \; {[a]}_1
\end{array}\right),\ \ a^{\top}=\left(\begin{array}{c} \; \;
\, [a]_2 \\ -{[a]}_1 \end{array}\right),
$$
and for $a,b\in{\rm R}^2$

$$
\left(\begin{array}{c} a \\ b
\end{array}\right)=\left(\begin{array}{cc} {[a]}_1 & {[a]}_2 \\
{[b]}_1 & {[b]}_2 \end{array}\right),\ \ (a\
b)=\left(\begin{array}{cc} {[a]}_1 & {[b]}_1 \\ {[a]}_2 & {[b]}_2
\end{array}\right).
$$
Moreover, $y=y(t)$ will denote the solution of system (\ref{ls})
linearly independent with $\dot x_0(t)$ and $(s,\theta)\to
F(s,\theta)$ is the following function
$$
F(s,\theta)=\int\limits_{s-T}^s \left( \dot x_0(\tau)\ y(\tau)
\right)^{-1}\phi(\tau-\theta,x_0(\tau))d\tau.
$$

\vspace{3mm}\noindent We can prove the following.

\begin{theorem}\label{th3} Let $f:[0,T_0]\to{\rm R}^2$ be a function satisfying the
following conditions

{\rm (B$_2$)} $\left<F(s,\theta),f(\theta)\right>\not=0,\ \
\mbox{for any} \ s\in[0,T],\ \theta\in [0,T_0],$

{\rm (B$_3$)} ${\rm deg}(N,[0,T_0])\not=1,$
\\
where $N(\theta)=\left({y(\theta)}^{\top} \ {\dot
x_0(\theta)}^\bot\right)f(\theta), \; \theta\in [0,T_0]$. Then
assumptions {\rm (A$_2$)} and {\rm (A$_3$)} of theorem 1 are
satisfied.
\end{theorem}

\vspace{2mm} \noindent To prove this theorem we need the following
preliminary lemma.

\begin{lemma}\label{lemma3}
We have that
\begin{eqnarray}
& \left<\eta(T,s,x_0(\theta))-\eta(0,s,x_0(\theta)),N(\theta)
\right>=\\
 & =  \int\limits_{s-T+\theta}^{s+\theta}  \left<d(\tau,0)
    \phi(\tau-\theta,x_0(\tau)),\,
  <\dot{x_0}(\theta),N(\theta)>{y(\tau)}^{\top}+ <y(\theta),N(\theta)>{\dot{x_0}(\tau)}^{\bot} \right>
  d\tau,\nonumber
\end{eqnarray}
\end{lemma}
where
$$
d(t,\theta)=\left({\rm det}\left(\begin{array}{l} y(t+\theta)^{\top} \\
\dot x_0(t+\theta)^{\bot}\end{array}\right)\right)^{-1}.
$$

\begin{proof} Observe, that
$$
\Omega'_{(3)}(0,t,\Omega(t,0,x_0(\theta)))={K(0,\theta)\left(K(t,\theta)\right)}
^{-1},
$$
where $K(\cdot,\theta)$ is a matrix whose columns are linearly
independent solutions of the differential system
\begin{equation}\label{sis2}
\dot x=\psi'(x_0(t+\theta))x.
\end{equation}
Let us choose the matrix $K(\cdot,\theta)$ as follows
$$
  K(t,\theta)=\left(\dot {x_0}(t+\theta)\ \ y(t+\theta)\right).
$$
Thus
\begin{eqnarray}
& &   \Omega'_{(3)}(0,t,\Omega(t,0,x_0(\theta)))=\left(\dot
{x_0}(\theta)\ \ y(\theta)\right)
   {\left(\dot x_0(t+\theta)\ \ y(t+\theta)\right)}^{-1}=\nonumber\\
& & =\left(\dot{x_0}(\theta)\ \
y(\theta)\right)\left(\begin{array}{lll}
        {y(t+\theta)}^{\top} \\
    {\dot{x_0}(t+\theta)}^{\bot}
    \end{array}\right) d(\tau,\theta). \nonumber
\end{eqnarray}

\vspace{3mm} \noindent Furthermore, we have
\begin{eqnarray}
& & \left<\eta(T,s,x_0(\theta))-\eta(0,s,x_0(\theta)),N(\theta)
\right>=
\nonumber \\
& &
=\left<\,\int\limits_{s-T}^s\Omega'_{(3)}(0,\tau,\Omega(\tau,0,x_0(\theta)))
\phi(\tau,x_0(\tau+\theta))d\tau,N(\theta)    \right>= \nonumber \\
& & =\left<\, \int \limits_{s-T}^s
d(\tau,\theta)\left(\dot{x_0}(\theta)\ y(\theta)\right)
    \left(\begin{array}{lll}
        {y(\tau+\theta)}^{\top} \\
    {\dot{x_0}(\tau+\theta)}^{\bot}
    \end{array}\right) \phi(\tau,x_0(\tau+\theta))d\tau,N(\theta)    \right>=
\nonumber \\
& & =\left<\,\int\limits_{s-T}^s d(\tau,\theta)
    \left(\begin{array}{lll}
        {y(\tau+\theta)}^{\top} \\
    {\dot{x_0}(\tau+\theta)}^{\bot}
    \end{array}\right) \phi(\tau,x_0(\tau+\theta))d\tau,
        \left(\begin{array}{lll}
            \dot{x_0}(\theta) \\
            y(\theta)
    \end{array}\right) N(\theta)    \right>= \nonumber \\\nonumber\\
& &  =\left<\,\int\limits_{s-T}^s d(\tau,\theta)
    \left(\begin{array}{lll}
        {y(\tau+\theta)}^{\top} \\
    {\dot{x_0}(\tau+\theta)}^{\bot}
    \end{array}\right) \phi(\tau,x_0(\tau+\theta))d\tau,
\left(\begin{array}{lll}
            <\dot{x_0}(\theta),N(\theta)> \\
            <y(\theta),N(\theta)>
    \end{array}\right)     \right>= \nonumber \\
& &  =\int\limits_{s-T}^s \left<d(\tau,\theta)
    \phi(\tau,x_0(\tau+\theta)), ({y(\tau+\theta)}^{\top} \ \
{\dot{x_0}(\tau+\theta)}^{\bot})
     \left(\begin{array}{lll}
            <\dot{x_0}(\theta),N(\theta)> \\
            <y(\theta),N(\theta)>
    \end{array}\right)     \right> d\tau= \nonumber \\
&=&\int\limits_{s-T}^s \left<d(\tau,\theta)
    \phi(\tau,x_0(\tau+\theta)), <\dot{x_0}(\theta),N(\theta)>{y(\tau+\theta)}^{\top}
              +  <y(\theta),N(\theta)>{\dot{x_0}(\tau+\theta)}^{\bot}
     \right> d\tau=\nonumber\\
& &  =\int\limits_{s-T+\theta}^{s+\theta}\left<d(\tau,0)
    \phi(\tau-\theta,x_0(\tau)), <\dot{x_0}(\theta),N(\theta)>{y(\tau)}^{\top}
              +  <y(\theta),N(\theta)>{\dot{x_0}(\tau)}^{\bot}
     \right> d\tau.\nonumber
\end{eqnarray}
\end{proof}

\noindent We can now prove Theorem 3.

\vspace{3mm} \noindent
\begin{proof}
By lemma 3 and the fact that
$$
\left< \dot x_0(\theta),y(\theta)^\top\right>= \left<
y(\theta),{\dot x_0(\theta)}^\bot\right>,
$$
we have
\begin{eqnarray}
 & & \left<\eta(T,s,x_0(\theta))-\eta(0,s,x_0(\theta)),N(\theta)\right>=
\nonumber \\
 & &
\int\limits_{s-T+\theta}^{s+\theta}\left<d(\tau,0)\phi(\tau-\theta,x_0(\tau)),{\left[f(\theta)\right]}_1
\left<\dot x_0(\theta),y(\theta)^\top\right> y(\tau)^{\top}+\right. \nonumber\\
 & & +{\left[f(\theta)\right]}_2\left<\dot x_0(\theta),{\dot
x_0(\theta)}^\bot\right> y(\tau)^{\top}
+{\left[f(\theta)\right]}_1
\left<y(\theta),y(\theta)^\top\right> {\dot x_0(\tau)}^\bot +\nonumber \\
 & & \left.+{\left[f(\theta)\right]}_2\left<y(\theta),{\dot
x_0(\theta)}^\bot\right> {\dot x_0(\tau)}^\bot\right>d\tau=\nonumber \\
 & & =\int\limits_{s-T+\theta}^{s+\theta}\left<d(\tau,0)\phi(\tau-\theta,x_0(\tau)),\left<\dot
x_0(\theta),y(\theta)^\top\right>\left({\left[f(\theta)\right]}_1\,
y(\tau)^{\top}+{\left[f(\theta)\right]}_2\, {\dot
x_0(\tau)}^\bot\right)\right>=\nonumber\\
 & & =\left<\dot x_0(\theta)\
y(\theta)^\top\right>\int\limits_{s-T+\theta}^{s+\theta}\left<d(\tau,0)\phi(\tau-\theta,x_0(\tau)),\left(
y(\tau)^{\top}\ {\dot x_0(\tau)}^\bot\right) f(\theta)\right> d\tau=\nonumber\\
 & & =\left<\dot x_0(\theta)\
y(\theta)^\top\right>\int\limits_{s-T+\theta}^{s+\theta}\left<d(\tau,\theta)\left(\begin{array}{c}
y(\tau)^{\top} \\ {\dot x_0(\tau)}^\bot \end{array}\right)
\phi(\tau-\theta,x_0(\tau)),f(\theta))\right>d\tau.\nonumber
\end{eqnarray}
Therefore, by the condition (B$_2$)
\begin{equation}\label{111}
  \left<\eta(T,s,x_0(\theta))-\eta(0,s,x_0(\theta)),N(\theta)\right>\not=0,\quad \mbox{for any} \quad s\in[0,T],\
\theta\in[0,T_0],
\end{equation}
that is the condition (A$_2$) of Theorem 1 hold true. Condition
(A$_3$) of theorem 1 also follows from (\ref{111}) by taking into
account condition (B$_3$) and lemma 1.
\end{proof}

\section{The case when $\dfrac{T_0}{T_1}$ is irrational}

\noindent In this section we assume the following condition

\vspace{3mm} \noindent ($A_1'$)-$\;$ $\dfrac{T_0}{T_1}$ is an
irrational number.

\vspace{5mm} \noindent We can prove the following result.

\begin{theorem}\label{th4}
Assume that there exists $\xi\in\tilde x_0$ such that the function
$\phi(t,\xi)$ is not constant with respect to $t.$ Assume
($A_1'$), then system (\ref{2}) does not have $T-$periodic
solutions  $x_\varepsilon$, $\varepsilon\in(0,\hat\varepsilon)$,
converging to $\tilde x_0$ as $\varepsilon\to 0,$ whenever $T>0$
and $\hat\varepsilon>0$.
\end{theorem}

\begin{proof}
Assume the contrary, thus there exists $T>0$ and
$\hat\varepsilon>0$ such that for every $\varepsilon\in (0,
\hat\varepsilon)$ system (\ref{2}) has $T$-periodic solution
$x_\varepsilon(t)$ satisfying
$$
  \lim_{\varepsilon\to 0} x_\varepsilon(t)=x_0(t+w_0),
$$
for some $w_0\in[0,T_0]$. Hence $T$ is a period for the function
$x_0.$ On the other hand, since $T_0$ is the least period of
$x_0,$ we have that $T=n_0 T_0$ for some $n_0\in{\rm N}.$ Moreover
in (\ref{2}) we have that the function $t\to
\phi(t,x_\varepsilon(t))$ is $T$-periodic, since $\dot x(t)$ is
$T$-periodic, thus, for any $t_0\in[0,T]$, we get
$$
\phi(t_0,x_\varepsilon(t_0))=\phi(t_0+p\, n_0
T_0,x_\varepsilon(t_0+p\, n_0 T_0))=
\phi(t_0+T_p,x_\varepsilon(t_0)),\ \mbox{whenever} \ p\in{\rm N},
$$
where $T_p= p\, n_0 T_0 ({\rm\  mod\ } T_1).$ Condition ($A_1'$)
implies that
$$
\overline{\bigcup_{p\in{\rm N}} T_p}=[0,T_1]
$$
and so
$$
\phi(t,x_\varepsilon(t_0))=c_\varepsilon,\ \mbox{for any} \
t\in[0,T_1],
$$
with $c_\varepsilon\in{\rm R}^2.$ By letting $\varepsilon\to 0$ in
the previous relation we obtain a contradiction with the
assumption that $t\to \phi(t, \xi)$ is not constant for at least
one $\xi\in \tilde x_0,$ in fact  $t_0$ is any point of $[0, T],$
and $T=n_0T_0, n_0\in {\rm N}$.
\end{proof}

\section{Example}\label{pr}

In this section we consider the following illustrative example for
system (\ref{1}).
\begin{equation}\label{e1}
  \begin{array}{lll}
    \dot x_1 & = & x_2-x_1(x_1^2+x_2^2-1) \nonumber\\
    \dot x_2 & = & -x_1-x_2(x_1^2+x_2^2-1). \nonumber
  \end{array}
\end{equation}
It is easy to see that system (\ref{e1}) has the limit circle
$x_0(\theta)=(\sin\theta,\cos\theta)$ with period $T_0=2\pi.$ In
order to verify the conditions of theorem 1 we use theorem 3.
Thus, we consider
\begin{eqnarray}
\dot x_0(t)=\left(\begin{array}{l} \quad \cos t\\
-\sin t\end{array}\right), \nonumber
\end{eqnarray}
It is easy to verify that $y(t)={\rm e}^{-2t}(\sin t,\cos t)$
satisfies the linearized system (\ref{ls}) corresponding to system
(\ref{e1}). Therefore
\begin{eqnarray}
 & &  y(t)^\bot={\rm e}^{-2t}\left(\begin{array}{l} -\cos t\\
\quad \sin t\end{array}\right)\nonumber\\
 & &  y(t)^\top={\rm e}^{-2t}\left(\begin{array}{l} \quad \cos t\\
-\sin t\end{array}\right).\nonumber
\end{eqnarray}
Define the function $f:[0, 2\pi]\to (\rm R)^2$ by the formula
$$
  f(\theta)=\left(\begin{array}{l} \sin\theta\\
\cos\theta\end{array}\right).
$$
If $t\to\phi(t,x)$ is a $4\pi$-periodic function, then condition
(B$_2$) of theorem 3 takes the form
\begin{equation}\label{us2}
\left<\,\int\limits_{s-4\pi+\theta}^{s+\theta}\left(\begin{array}{cc}
\cos\tau & {\em e}^{-2\tau}\sin\tau \\ -\sin\tau & {\rm
e}^{-2\tau}\cos\tau\end{array}\right)^{-1}\phi(\tau-\theta,x_0(\tau))d\tau,\left(\begin{array}{c
} \sin\theta\\ \cos\theta\end{array}\right)\right>\not=0
\end{equation}
for any $s\in[0,4\pi],\ \theta\in [0,2\pi].$

\noindent As example of function $\phi$ satisfying condition
(\ref{us2}) we consider
\begin{equation}\label{ex_phi}
\phi(t,\xi)=\left(\begin{array} {cc} \xi_2 & \xi_1 \\ -\xi_1 &
\xi_2\end{array}\right)\left(\begin{array}{c}\xi_1\cos t-\xi_2\sin t +a\sin\frac{t}{2} \\
\xi_1\sin t+\xi_2\cos t \end{array}\right),
\end{equation}
with  $a>0$ sufficiently small. In fact, for $a=0$ we have
\begin{eqnarray}
& &
\left<\,\int\limits_{s-4\pi+\theta}^{s+\theta}\left(\begin{array}{cc}
\cos\tau & {\em e}^{-2\tau}\sin\tau \\ -\sin\tau & {\rm
e}^{-2\tau}\cos\tau\end{array}\right)^{-1}\phi(\tau-\theta,x_0(\tau))d\tau,\left(\begin{array}{c
} \sin\theta\\ \cos\theta\end{array}\right)\right>=\nonumber\\
& & =
\left<\,\int\limits_{s-4\pi+\theta}^{s+\theta}\left(\begin{array}{cc}
\cos\tau & -\sin\tau \\ {\rm e}^{2\tau}\sin\tau & {\rm
e}^{2\tau}\cos\tau\end{array}\right) \left(\begin{array}{ll}
\cos\tau & \sin\tau \\ -\sin\tau & \cos\tau\end{array}\right)
\cdot \right.\nonumber\\
& & \left.\cdot\left( \begin{array}{l}
\sin\tau\cos(\tau-\theta)-\cos\tau\sin(\tau-\theta) \\
\sin\tau\sin(\tau-\theta)+\cos\tau\cos(\tau-\theta)
\end{array}\right)
d\tau,\left(\begin{array}{l} \sin\theta\\
\cos\theta\end{array}\right)\right> \nonumber\\
& & =  \left<\left(\sin\theta \
\cos\theta\right)\int\limits_{s-4\pi+\theta}^{s+\theta}\left(\begin{array}{cc}
1 & 0
\\ 0 & {\rm
e}^{2\tau}\end{array}\right) d\tau ,\left(\begin{array}{l} \sin\theta\\
\cos\theta\end{array}\right)\right>=\nonumber\\
& & =4\pi
\sin^2\theta+\cos^2\theta\int\limits_{s-4\pi+\theta}^{s+\theta}
{\rm e}^{2\tau}d\tau>0
\end{eqnarray}
and so (\ref{us2}) holds for the function (\ref{ex_phi}) for any
$a>0$ sufficiently small.

\noindent Let us now verify condition (B$_3$) of theorem 3. For
this, consider
$$
N(\theta)=\left(\begin{array}{cc} {\rm e}^{-2\theta}\cos\theta & \sin\theta \\
-{\rm e}^{-2\theta}\sin\theta &
\cos\theta\end{array}\right)\left(\begin{array}{c} \sin\theta\\
\cos\theta\end{array}\right).
$$
and observe that the homotopy
$$
  N_\lambda(\theta)=\left(\begin{array}{cc} {\rm e}^{-2\lambda\theta}\cos\theta
& \sin\theta \\ -{\rm e}^{-2\lambda\theta}\sin\theta &
\cos\theta\end{array}\right)\left(\begin{array}{c} \sin\theta\\
\cos\theta\end{array}\right),
$$
joining the vector fields $N_0$ and $N_1=N$ does not vanish
whenever $\lambda\in[0,1].$ Therefore
$$
  {\rm deg}(N,[0,2\pi])={\rm deg}(N_0,[0,2\pi])={\rm deg}\left(\left(\begin{array}{c}
\sin (2\cdot) \\ \cos(2\cdot) \end{array}\right),[0,
2\pi]\right)=2
$$
and so condition (B$_3$) of theorem 3 is also satisfied.

\noindent In conclusion, for $0<\varepsilon<\varepsilon_\gamma,$
where $\varepsilon_\gamma>0$ is given as in (\ref{est}), the
following system
$$
  \left(\begin{array}{l} \dot x_1 \\ \dot x_2 \end{array}\right)=\left(\begin{array}{l}
  x_2 -x_1(x_1^2+x_2^2-1)\\ -x_1-x_2(x_1^2+x_2^2-1)\end{array}\right)+
$$
$$
+\varepsilon \left(\begin{array}{cc} x_2 & x_1 \\ -x_1 &
x_2\end{array}\right)\left(\begin{array}{l}x_1\cos t-x_2\sin t
+a\sin\frac{t}{2} \\ x_1\sin t+x_2\cos t \end{array}\right),
$$
where $a>0$ is sufficiently small, has a $4\pi$-periodic solutions
$x_{\varepsilon,1}$ and $x_{\varepsilon,2}$ which converge to the
unitary circle from the outside and from the inside respectively
as $\varepsilon\to 0.$



\begin{thebibliography}{100}

\bibitem{AGL}
J. Andres, L. G\' orniewicz and M. Lewicka, Partially dissipative
periodic processes, in ``Topology in Nonlinear Analysis'', Banach
Center Publications, {\bf 35}, Institute of Mathematics Polish
Academy of Sciences, Warszawa (1996).

\bibitem{b}
J. Baillieul, Stable average motions of mechanical systems subject
to periodic forcing, in ``Dynamics and Control of Mechanical
Systems: The Falling Cat and Related Problems'', Fields Inst.
Commun., Vol. 1, M. J. Enos, ed., AMS, Providence, RI, (1993),
1-23.

\bibitem{B1}
L. B. Bushard, Periodic solutions of perturbed autonomous systems,
and locking-in, SIAM J. Appl. Math., Vol. 22 (1972), 519-528.

\bibitem{B2}
L. B. Bushard, Behavior of the periodic surface for a periodically
perturbed autonomous system and periodic solutions, J.
Differential Equations, Vol. 12 (1972), 487-503.

\bibitem{B3}
L. B. Bushard, Periodic solutions and locking-in on the periodic
surface, Intern. J. Non-Linear Mech., Vol. 8 (1973), 129-141.

\bibitem{bog}
N. N. Bogolyubov and  Yu.A. Mitropol'skii, ``Asymptotic methods in
the theory of non-linear oscillations'', Moscow, Fizmatgiz, 1963.

\bibitem{bullo}
F. Bullo, Averaging and vibrational control of mechanical systems,
SIAM J. Control Optim., Vol. 41 (2002), 542-562.

\bibitem{CMZ}
A. Capietto, J. Mawhin and F. Zanolin, Continuation theorems for
periodic perturbations of autonomous systems, Trans. Amer. Math.
Soc., Vol. 329 (1992), 41-72.

\bibitem{F1}
M. Farkas, Controllably periodic perturbations of autonomous
systems, Acta Math. Acad. Sci. Hungar, Vol. 22 (1971/72), 337-348.

\bibitem{F2}
M. Farkas, Periodic perturbations of autonomous systems, Alkalmaz.
Mat. Lapok, Vol. 1 (1975), 197-254.

\bibitem{Furi}
M. Furi, M. P. Pera, M. Spadini, Forced oscillations on manifolds
and multiplicity results for periodically perturbed autonomous
systems, J. Comput. Appl. Math., Vol. 113 (2000), 241-254.

\bibitem{Furi1}
M. Furi, M. Spadini, Branches of Forced Oscillations for
Periodically Perturbed Second Order Autonomous ODEs on Manifolds,
J. Differential Equations, Vol. 154 (1999), 96-106.

\bibitem{G1}
K. Gopalsamy, Limit cycles in periodically perturbed population
systems, Bull. Math. Biol., Vol. 22 (1981), 463-485.

\bibitem{kmn1}
M. I. Kamenskii, O. Yu. Makarenkov, and P. Nistri, A new approach
in the theory of ordinary differential equations with small
parameter, Doklady Akademii Nauk, Vol. 388 (2003), 439-442.

\bibitem{kmn2}
M. I. Kamenskii, O. Yu. Makarenkov, and P. Nistri, Periodic
solutions for a class of singularly perturbed systems, Dyn.
Contin., Discrete Impuls. Systems, Ser. A: Math. Anal., Vol. 11
(2004), 41-55.

\bibitem{kmn3} M. Kamenskii, O. Yu. Makarenkov and P. Nistri, Small parameter
perturbations of nonlinear periodic systems, Nonlinearity, Vol. 17
(2004), 193-205.

\bibitem{KZ}
M. A. Krasnosel'skii, P. P. Zabreiko, ``Geometrical Methods of
Nonlinear Analysis'', Springer-Verlag, Berlin, 1984.

\bibitem{KPPZ}
M. A. Krasnosel'skii, A.I. Perov, A.I. Povolockii, P. P. Zabreiko,
``Plane vector fields'', Academic Press, New York, 1966.

\bibitem{Levinson}
N. Levinson, Transformation theory of nonlinear differential
equations of the second order, Ann. of Math., Vol. 45 (1950),
723-737.

\bibitem{L1}
W.S. Loud, Periodic solutions of perturbed autonomous systems,
Ann. of Math., Vol. 70 (1959), 490-529.

\bibitem{L2}
W.S. Loud, The location of the invariant manifold for a perturbed
autonomous systems, J. Math. and Phys., Vol. 40 (1961), 87-102.

\bibitem{L3}
W.S. Loud, Periodic solutions of perturbed second-order autonomous
equations, Mem. Amer. Math. Soc., Vol. 47, (1964), pp. 133.

\bibitem{sch}
K. R. Schneider, Vibrational control of singularly perturbed
systems, in Lecture Notes in Control and Information Science Vol.
259, Springer Verlag, London (2001), 397-408.

\end{thebibliography}
\end{document}